\documentclass{amsart}

\usepackage{hyperref}
\hypersetup{
    colorlinks=true,
    linkcolor=blue,
    filecolor=blue,
    urlcolor=blue,
    citecolor=blue,
}
\usepackage[shortalphabetic]{amsrefs}
\usepackage{amsmath}
\usepackage{amsthm}
\usepackage{amssymb}

\numberwithin{figure}{section}

\numberwithin{equation}{section}
\newtheorem{thm}[equation]{Theorem}
\newtheorem*{thm*}{Theorem}
\newtheorem{lem}[equation]{Lemma}
\newtheorem{prop}[equation]{Proposition}

\theoremstyle{definition}
\newtheorem{defn}[equation]{Definition}
\newtheorem{remark}[equation]{Remark}
\newtheorem{ex}[equation]{Example}

\DeclareMathOperator{\End}{End}
\DeclareMathOperator{\Hom}{Hom}

\DeclareMathOperator{\Aut}{Aut}

\DeclareMathOperator{\GL}{GL}

\DeclareMathOperator{\PGL}{PGL}
\DeclareMathOperator{\PSL}{PSL}

\DeclareMathOperator{\GammaL}{\Gamma L}

\DeclareMathOperator{\Gal}{Gal}
\DeclareMathOperator{\Stab}{Stab}

\DeclareMathOperator{\im}{im}
\DeclareMathOperator{\Pf}{Pf}

\renewcommand{\phi}{\varphi}

\begin{document}

\title{Enumerating isoclinism classes of semi-extraspecial groups}
\author{Mark L. Lewis}
\address{
	Kent State University\\
	Department of Mathematical Sciences\\
  MSB 233\\
	Kent, OH 44242
}
\email{lewis@math.kent.edu}

\author{Joshua Maglione}
\address{
	Kent State University\\
	Department of Mathematical Sciences\\
  MSB 233\\
	Kent, OH 44242
}
\email{jmaglion@math.kent.edu}

\date{\today}
\keywords{Pfaffian, genus 2, bilinear maps}

\begin{abstract}
We enumerate the number of isoclinism classes of semi-extraspecial $p$-groups with derived subgroup of order $p^2$.
To do this, we enumerate $\GL(2, p)$-orbits of sets of irreducible, monic polynomials in $\mathbb{F}_p[x]$.
Along the way, we also provide a new construction of an infinite family of semi-extraspecial groups as central quotients of Heisenberg groups over local algebras.
\end{abstract}

\maketitle

\section{Introduction}

All groups in this paper are finite, and $p$ is a prime.
A $p$-group $G$ is said to be {\it semi-extraspecial} if every subgroup $N$ having index $p$ in $Z (G)$ satisfies the condition that $G/N$ is extra-special.
Semi-extraspecial groups are a special case of Camina groups, i.e.\ groups $G$ where every irreducible character vanishes on $G\backslash G'$ \cites{Camina, Macdonald:semi-extraspecial, CMS:Frobenius-groups}.
In fact, semi-extraspecial groups are exactly the Camina $p$-groups of class 2 \cite{Verardi}*{Theorem~1.2}.

These groups are virtually indistinguishable, so in the context of group isomorphism and enumeration, they pose a challenge.
Indeed, the (nontrivial) central quotients of Heisenberg groups over fields, considered in \cite{LW:Heisenberg}, are all semi-extraspecial groups.
Those quotients have exponent $p$, isomorphic character tables, two possible orders of centralizers, and a limited number of possible isomorphism types of automorphism groups---after fixing the order of the group and derived subgroup.

Applying ideas of Brooksbank-Maglione-Wilson \cite{BMW}, we the study of semi-extraspecial $p$-groups $G$ with $|G'| = p^2$ by considering a $\PGL(2,p)$-action on multisets of polynomials in $\mathbb{F}_p[x]$. 
In particular, we enumerate the number of isoclinism classes of such groups up to order $p^{12}$.
A slight extension of Verardi \cite{Verardi}*{Corollary~5.11} states that there is a unique isoclinism class of semi-extraspecial groups $G$ with $|G| = p^6$ and $|G'| = p^2$ for each prime $p$.
We show that a similar result holds for the next larger size of semi-extraspecial groups.

\begin{thm}\label{main1}
For each prime $p$, there is a unique isoclinism class of semi-extraspecial groups $G$ with $|G| = p^8$ and $|G'| = p^2$.
\end{thm}

At this point, one might be tempted to conjecture for every prime $p$ and every positive integer $n$
that there is a unique isoclinism class of semi-extraspecial groups $G$ with $|G| = p^{2n+2}$ and $|G'| = p^2$.
However, starting with the next size, we will see that the situation becomes more complicated.

\begin{thm}\label{main2}
For every prime $p$, there are $p + 3 - \gcd(2, p)$ isoclinism classes of semi-extraspecial groups $G$ with $|G| = p^{10}$ and $|G'| = p^2$.
\end{thm}

For the next size, the formula gets even more complicated.

\begin{thm} \label{main3}
For every prime $p$, the number of isoclinism classes of semi-extraspecial groups $G$ with $|G|=p^{12}$ and $|G'|=p^2$ is
\[ \frac{1}{30}\left( 11p^2 - 5p - 22 +10\gcd(3,p-2) + 6\gcd(5,p) + 12\gcd(5, p^2-1)\right).  \]
\end{thm}

We could continue in this fashion and find formulas, depending on $p$, for the number of isoclinism classes of semi-extraspecial groups $G$ with $|G| = p^{2n + 2}$ and $|G'| = p^2$ for values of $n$ that are larger than $5$.
However, these functions appear to get progressively more complicated as $n$ increases, and we believe that these cases illustrate the procedure so that the sufficiently motivated reader can compute the function for larger $n$ if desired.

\medskip

Along the way to enumerating these isoclinism classes, we find a new way to construct an infinite family of semi-extraspecial groups.
To date, all known constructions of semi-extraspecial groups begin with a semifield $A$ or by taking a central product of semi-extraspecial groups.  A (finite) semi-field is an algebra with identity where the product $a *b = 0$ if and only if either $a = 0$ or $b = 0$.  Since $A$ is finite, this implies that every nonzero element has both a left inverse and a right inverse.  However, since $A$ is not necessarily associative or commutative, these left and right inverses need not be equal.  It is this invertibility condition that makes the following groups semi-extraspecial.

If $A$ is any algebra, the Heisenberg group over $A$ is the group
\begin{equation}\label{eqn:semifield-group}
  H(A) = \left\{ \begin{bmatrix} 1 & e & z \\ & 1 & f \\ & & 1 \end{bmatrix} \;\middle|\; e,f,z\in A \right\}.
\end{equation}
It is not difficult to see that $H(A)$ is semi-extraspecial if and only if $A$ is a semifield \cite{Lewis:expository}.  Because $H(A)$ is semi-extraspecial when $A$ is a semifield, every central quotient is either abelian or semi-extraspecial.  One can twist the semifield group in~\eqref{eqn:semifield-group} to get a different class of semi-extraspecial groups \cite{Lewis:ses-groups}.  At this time, the only way we know how to construct a semi-extraspecial group is to take a quotient of $H(A)$ for a semifield $A$, a quotient of the twisted semifield groups in \cite{Lewis:ses-groups}, or to take central products of these groups.

To obtain the new construction in this paper, we start with an algebra $A$ with zero divisors.  Because of the zero divisors, $H(A)$ is not semi-extraspecial.
We correct this by constructing the quotient by a subspace that complements the minimal ideal in $A$.

\begin{thm}\label{main4}
Let $K$ be a finite field, $a(x)\in K[x]$ be irreducible of degree $\geq 2$, and let $c>1$ be an integer.
Suppose $S\leq A$ is a $K$-subspace, and set $G=H(A)/N(S)$, where
\[ N(S) = \left\{ \begin{bmatrix} 1 & 0 & s \\ & 1 & 0 \\ & & 1 \end{bmatrix} \;\middle|\; s\in S \right\} . \]
Then $G$ is semi-extraspecial if, and only if, $A= S + (a(x)^{c-1})$.
\end{thm}

We ask whether all of the groups in Theorem~\ref{main4} are also central quotients of the semifield groups in~\eqref{eqn:semifield-group}.  This is true for all of the ones we have been able to determine.  The case we have studied most closely are the quotients of the semifield groups of order $3^{12}$.  We note that although there are many infinite families of semifields that have been constructed \cites{JJB:Handbook, Kantor:finite-semifields}, there is no evidence that these families contain all of them. Indeed, they do not span the table of semifields of order $81$ \cite{Dempwolff:81}.
Going back to our initial question, we ask if $G$ is as in Theorem~\ref{main4}, then does there exist a semifield $A$ of order $p^{cd}$ such that $G \cong H(A)/Z$ where $Z$ is some proper subgroup of the center of $H (A)$?  It is possible that the construction in Theorem \ref{main4} might yield new semi-extraspecial groups that do not arise as a quotient of a group associated with any semifield.

\section{Groups and bilinear maps}

Assume that $K$ is a field and $U$, $V$, and $W$ are $K$-vector spaces.
A map $\circ : U\times V\rightarrowtail W$ is $K$-bilinear if for all $u,u'\in U$, $v,v'\in V$, and $k\in K$,
\begin{align*}
  (u+ku')\circ v &= u\circ v + k(u'\circ v), & u \circ (v+kv') &= u\circ v + k(u\circ v').
\end{align*}
We will use $\rightarrowtail$ to denote a bilinear map (bimap), and in the case where $\dim_K(W)=1$, we call $\circ$ a bilinear \emph{form}.
The \emph{radicals} of $\circ$ are $U^\perp=\{ v \in V \mid U\circ v = 0\}$ and $V^\perp=\{ u\in U\mid u\circ V = 0\}$.
The bimap $\circ$ is \emph{nondegenerate} if $U^\perp$ and $V^\perp$ are trivial, and $\circ$ is \emph{full} if $U\circ V = W$.
We say a bimap is \emph{fully nondegenerate} if it is both full and nondegenerate.

Our main source of bimaps come from finite $p$-groups, $G$, of exponent $p$ and class 2, via the Baer correspondence.
For such a group we define the biadditive commutator map to be $\circ_G : G/Z(G)\times G/Z(G) \rightarrowtail G'$, where
\[ (Z(G)g)\circ (Z(G)h) = [g,h]. \]
Then $\circ_G$ is $\mathbb{F}_p$-bilinear and fully nondegenerate.
In the case when $G$ is semi-extraspecial, both $G/Z(G)$ and $G'$ are elementary abelian, and we write $\circ_G : V\times V\rightarrowtail W$, where $V=G/Z(G)$ and $W=G'$.

Two bimaps $\circ, \bullet : V\times V \rightarrowtail W$ are \emph{pseudo-isometric} if there exists $(f,g)\in\Aut(V)\times \Aut(W)$ such that for all $u,v\in V$, $(uf)\bullet (vf) = (u\circ v)g$.
The bimaps are \emph{isometric} if they are pseudo-isometric with pseudo-isometry $(f,1)$.
In \cite{Hall:isoclinism}, Hall initiated a study of $p$-groups up to isoclinism as a weaker form of isomorphism.
Two class 2 groups $G$ and $H$ are \emph{isoclinic} if there exists isomorphisms between $f : G/Z(G)\rightarrow H/Z(H)$ and $g : G'\rightarrow H'$ such that $(f,g)$ is a pseudo isometry from $\circ_G$ to $\circ_H$.

\begin{defn}
A class $2$ $p$-group $G$ is \emph{semi-extraspecial} if $G'=Z(G)=G^p$ and if for all $g\in G\backslash Z(G)$ and $z\in G'$, there exists $h\in G$ such that $[g,h]=z$.
\end{defn}

In the context of bimaps, $\circ:V\times V\rightarrowtail W$, this is equivalent to the following.
For all $u\in V$ the linear maps $L_u, R_u : V\rightarrow W$, where $vL_u= u\circ v$ and $vR_u = v\circ u$, are surjective.
As our bimaps are alternating, it is enough to check this for just
This implies that $\dim(W)\leq\dim(V)$.

An important algebra associated to a $K$-bimap $\circ: V \times V\rightarrowtail W$ is its centroid: the largest algebra $C$ for which is $C$-bilinear.
For $\Omega = \End(U)\times \End(V) \times \End(W)$, the centroid is
\begin{align*}
  \mathcal{C}(\circ) &= \{ (X, Y, Z)\in\Omega \mid (uX)\circ v = u\circ (vY) = (u\circ v)Z \} .
\end{align*}
If $\circ$ is fully nondegenerate, then $\mathcal{C}(\circ)$ is commutative.
The centroid plays a critical role in direct decompositions of groups \cite{Wilson:direct-decomposition} and in the definition of genus \cite{BMW}.

\subsection{Pfaffians of bimaps}\label{sec:Pfaffians}

Fix a $K$-bimap, $\circ : V\times V\rightarrowtail W$.
For $\omega\in W^*:=\Hom(W,K)$, define $\circ^\omega:U\times V\rightarrowtail K$ where
\[ u\circ^\omega v = (u\circ v)\omega.\]
Fix a basis $\mathcal{X}$ for $V$ and $\mathcal{Z}^*$ for $W^*$.
For each $\omega\in\mathcal{Z}^*$, write $\circ^\omega$ as a matrix $M_\omega$ with respect to $\mathcal{X}$, the structure constants of the $K$-form $\circ^\omega$.
The \emph{generalized discriminant} of $\circ: V\times V\rightarrowtail W$, with respect to $\mathcal{X}$ and $\mathcal{Z}^*$, is a polynomial in indeterminants $x_\omega$, for $\omega\in\mathcal{Z}^*$, where
\[ \text{disc}(\circ) = \det\left(\sum_{\omega\in\mathcal{Z^*}}M_\omega x_\omega\right). \]
When $\circ$ is \emph{alternating}, i.e.\ for all $u\in V$, $u\circ u=0$, define the (\emph{generalized}) \emph{Pfaffian} of $\circ$ to be the homogeneous polynomial $\Pf(\circ)$ such that $\Pf(\circ)^2=\text{disc}(\circ)$.

The next theorem is due to I.D.\ Macdonald \cite{Macdonald:semi-extraspecial} who proved the conclusion for Camina groups.

\begin{thm}[\cite{Macdonald:semi-extraspecial}*{Theorem~3.1}]\label{thm:Pfaffian}
A Pfaffian of $\circ_G:V\times V\rightarrowtail W$ has no solutions in $\mathbb{P}W$ if, and only if, $G$ is semi-extraspecial.
\end{thm}

\section{Decompositions, groups of genus 2, and Pfaffians}

A \emph{central decomposition} of a group $G$ is a set $\mathcal{H}=\{H : H\leq G\}$ such that $\langle \mathcal{H}\rangle = G$, for all $H\in\mathcal{H}$, $\langle \mathcal{H}-H\rangle \ne G$, and $[H,\langle\mathcal{H}-H\rangle]=1$.
Note that this definition allows for $Z(H_1)\ne Z(H_2)$ for $H_1,H_2\in\mathcal{H}$.
A group $G$ is \emph{centrally indecomposable} if $\{ G\}$ is the only central decomposition of $G$.
A \emph{direct decomposition} of $G$ is a central decomposition $\mathcal{H}$ such that for all $H\in\mathcal{H}$, $H\cap \langle \mathcal{H}-H\rangle=1$, and $G$ is \emph{directly indecomposable} if $\{ G\}$ is the only direct decomposition of $G$.
We say a central (or direct) decomposition $\mathcal{H}$ is \emph{fully-refined} if for all $H\in\mathcal{H}$, $H$ is centrally (or directly) indecomposable.

\begin{lem}\label{lem:direct-decomp}
If $G$ is semi-extraspecial, then $G$ is directly indecomposable.
\end{lem}

\begin{proof}
Suppose $\mathcal{H}$ is a direct decomposition of $G$ with $H,K\in\mathcal{H}$.
By definition, $[G,H]\leq H$, but since $G$ is semi-extraspecial, it follows that for all $z\in Z(K)$, there exists $g\in G$ such that $[g,h]=z$.
Therefore, $H=K$, so $\mathcal{H}=\{G\}$.
\end{proof}

The central product of $G$ with $H$ can be defined as a central quotient of $G\times H$ with an isomorphism $\theta: Z(G)\rightarrow Z(H)$, c.f.\ Hall \cite{Hall:isoclinism}.
We note that different isomorphisms $\theta$ may produce non-isomorphic central products of groups.
Our definition above allows for the case where $Z(G)$ may properly embed into $Z(H)$.
Because of this, we prove that with semi-extraspecial groups, we must have all the centers of each central factor isomorphic.
The following proposition is a slight generalization of Beisiegel \cite{Beisiegel:ses-groups}*{Lemma~2}.

\begin{prop}\label{prop:central-decomp}
Suppose $\mathcal{H}$ is a central decomposition of $G$ such that each $H\in\mathcal{H}$ is semi-extraspecial.
Then $G$ is semi-extraspecial if, and only if, for each $H\in\mathcal{H}$, $Z(H)=Z(G)$.
\end{prop}

\begin{proof}
First, we suppose that for all $H\in\mathcal{H}$, $Z(H)=Z(G)$.
If $N<Z(G)$ is a maximal subgroup, then $N$ is a maximal subgroup of each $Z(H)$.
Therefore, $H/N$ is extraspecial, and $G/N=\langle H/N : H\in\mathcal{H}\rangle$, where $Z(H/N)=Z(G/N)$.
This is a central decomposition of $G/N$ into extraspecial groups; hence, $G/N$ is extraspecial.

Conversely, suppose there exists $H\in\mathcal{H}$ such that $Z(H)\ne Z(G)$.
Let $N<Z(G)$ be a maximal subgroup.
If $Z(H)>Z(G)$, then $N$ is not a maximal subgroup of $Z(H)$, and in particular, $H/N$ is not extraspecial.
Therefore, $\langle HN/N: H\in\mathcal{H}\rangle$ cannot be extraspecial, so $G$ is not semi-extraspecial.
On the other hand, suppose there exists $z\in Z(G)$ such that $z\notin Z(H)=H'$.
Then for all $h\in H$ and $g\in H$, $[g,h]\ne z$.
From the central decomposition property, it follows that for all $h\in H$ and $g\in G$, $[g,h]\ne z$, so that $G$ is not semi-extraspecial.
\end{proof}

We state a useful theorem of Wilson~\cite{Wilson:direct-decomposition} characterizing direct decompositions of $p$-groups of class 2.
A commutative ring is \emph{local} if it has a unique maximal ideal.

\begin{thm}[\cite{Wilson:direct-decomposition}*{Theorem~8}]\label{thm:local-centroid}
Let $G$ be class $2$ and exponent $p$.
Then $\mathcal{C}(\circ_G)$ is a local $\mathbb{F}_p$-algebra if, and only if, $G$ is directly indecomposable.
\end{thm}

Lemma~\ref{lem:direct-decomp} together with Theorem~\ref{thm:local-centroid} are key to proving the following proposition.
In general for fully nondegenerate bimaps, the centroid is a commutative ring.
In the next lemma, we show that because $\mathcal{C}(\circ_G)$ is Artinian with trivial nilradical, $\mathcal{C}(\circ_G)$ must be a simple $\mathbb{F}_p$-algebra.

\begin{prop}\label{prop:centroid-field}
If $G$ is semi-extraspecial, then $\mathcal{C}(\circ_G)$ is a field.
\end{prop}

\begin{proof}
Suppose $(X,Z)\in\mathcal{C}(\circ_G)\subseteq \End(V)\times \End(W)$ is nilpotent.
For all $u\in V$, we claim that the linear map $L_{uX} : V\rightarrow W$, where $v\mapsto (uX)\circ v$, is not surjective.
In fact, since $(X,Z)\in\mathcal{C}(\circ_G)$, it follows that $\im(L_{uX})\leq\im(Z)$.
As $Z$ is nilpotent, $\im(Z)<W$, so for all $g\in G$, there exists $z\in Z(G)$ such that $z\notin [g,G]$.
This would imply that $G$ is not semi-extraspecial.
Therefore, the nilradical of $\mathcal{C}(\circ_G)$ is trivial.
Since $\mathcal{C}(\circ_G)$ is Artinian (in particular finite), it follows that $\mathcal{C}(\circ_G)$ is semi-simple.
By Lemma~\ref{lem:direct-decomp} and Theorem~\ref{thm:local-centroid}, it follows that $\mathcal{C}(\circ_G)$ is a simple $\mathbb{F}_p$-algebra.
Because $\circ_G$ is fully nondegenerate, $\mathcal{C}(\circ_G)$ is commutative, and by the Wedderburn--Artin Theorem, $\mathcal{C}(\circ_G)$ is a field.
\end{proof}

\begin{remark}\label{rem:cent}
Because $K=\mathcal{C}(\circ_G)$ is a field and $\circ_G$ is $K$-bilinear, we now assume throughout that $\circ_G:V\times V\rightarrowtail W$ is a bimap over $K$, so $V$ and $W$ are $K$-vector spaces.
Another fact that we will use later is that $\mathcal{C}(\circ_G)$ is faithfully represented on all three coordinates whenever $\circ_G$ is fully nondegenerate; in particular it is faithfully represented in $\End(W)$.
Indeed, if $(X,Z),(Y,Z)\in\mathcal{C}(\circ_G)\subseteq \End(V)\times\End(W)$, then for all $u,v\in V$, $(uX-uY)\circ_G v = 0$.
Because $\circ_G$ has trivial radical, $X=Y$.
\end{remark}

\begin{defn}
Suppose $G$ is a nilpotent class 2 and isoclinic to a direct product $H_1 \times \cdots \times H_r$ of directly indecomposable groups.
The \emph{genus} of $G$ is the maximum rank of $[H_i, H_i]$ as a $\mathcal{C}(\circ_{H_i})$-module.
\end{defn}

Lemma~\ref{lem:direct-decomp} and Proposition~\ref{prop:centroid-field} enable a simpler definition of genus for our purposes.
The genus of a semi-extraspecial group $G$ is the dimension of $G'$ as a $\mathcal{C}(\circ_G)$-vector space.
We will say a (semi-extraspecial) group $G$ is \emph{genus $d$ over $K$} if $\mathcal{C}(\circ_G)\cong K$ and $\dim_K(G')=d$.

\subsection{Groups of genus 2}

We cite two key theorems from \cite{BMW} that we use later.
Involved in the classification of groups of genus 2 are Heisenberg groups over local algebras, $A=K[x]/(a(x)^c)$, where $a(x)$ is irreducible in $K[x]$ and $c\geq 1$,
\begin{align*}
  H(A) &  =\left\{ \begin{bmatrix}
  1 & e & z \\ 0 & 1 & f \\ 0 & 0 & 1
  \end{bmatrix} \;\middle|\; \begin{array}{c} e,f,z\in A\end{array}\right\}.
\end{align*}
Observe that $H(A)$ is semi-extraspecial if, and only if, $c=1$.
To see this, when $c=1$, this is a Heisenberg group over a field which is semi-extraspecial, but on the other hand, if $c>1$, then the ring $K[x]/(a(x)^c)$ has nilpotent elements.
These nilpotent elements prevent the group from satisfying the property that for all $g\in G\backslash G'$ and $z\in G'$, there exists $h\in G$ such that $[g,h]=z$.
In the case when $c>1$, quotients of $H(A)$ can still be semi-extraspecial even though the group is not semi-extraspecial.
We will characterize these quotients in Section~\ref{sec:ses-groups-genus2}.

\begin{thm}[{\cite{BMW}*{Theorem~1.2}}]\label{thm:indecomp-genus2}
A centrally indecomposable $p$-group of exponent $p$ and genus $2$ over a finite field $K$ is isomorphic to one of the following
\begin{enumerate}
  \item[(i)] a central quotient of the Heisenberg group $H=H(K[x]/(a(x)^c))$ by a subgroup $N$, such that $1-N$ is a $K$-subspace of $1-H'$;

  \item[(ii)] the matrix group
  \begin{align*}
	  H^{\flat}(K, m) & :=
	  \left\{
	  \left[\begin{array}{c|c|c}
		  I_2 & \begin{array}{cccc}
			  e_1 & \cdots & e_m & 0 \\
			  0 & e_1 & \cdots & e_m
		  \end{array} &
		  \begin{array}{c} z_1 \\ z_2 \end{array}\\
		  \hline		
		   & I_{m+1} & \begin{array}{c} f_0\\ \vdots \\ f_{m}\end{array}\\
		   \hline
		   & & 1
	  \end{array}\right] \;\middle|\; e_i, f_i,z_i \in K\right\}.
%    H^\flat (K, m) & := \left\{ \left[\begin{array}{c|c|c}
%      1 & \begin{array}{ccc} e_0 & \cdots & e_m \end{array} & \begin{array}{cc} z_1 & z_2 \end{array} \\ \hline
%      & I_{m+1} & \begin{array}{cc} e_1 & 0 \\ \vdots & e_1 \\ e_m & \vdots \\ 0 & e_m \end{array} \\ \hline
%      & & I_2
%    \end{array} \right]\middle| e_i, f_i, z_i\in K \right\}.
  \end{align*}
\end{enumerate}
\end{thm}

A key idea used throughout the paper comes from the following theorem.
This converts the problem of enumerating groups to enumerating orbits of (principal ideals of) polynomials under the action of $\GammaL(2,K)=\GL(2,K)\rtimes \Gal(K)$.
We will describe the details of this action in Section~\ref{sec:GL2-action}.
For our purposes, the crux is that the isomorphism problem of semi-extraspecial groups of order $p^{2n+2}$ with exponent $p$ and derived subgroup of order $p^2$ is determined by the $\GL(2,p)$ action on the \emph{multiset} of Pfaffians of a fully refined central decomposition of $G$.
Specifically, if $\mathcal{H}$ is a fully-refined central decomposition of $G$, then the multiset of Pfaffians of $G$ is $\{ \Pf(\circ_H) \mid H\in\mathcal{H}\}$.
If every $H\in\mathcal{H}$ is isomorphic to a group from Theorem~\ref{thm:indecomp-genus2}(i), then this multiset completely determines the group.

\begin{thm}[\cite{BMW}*{Theorem~3.22}]\label{thm:Pfaffian-bijection}
Suppose $G_1$ and $G_2$ are $p$-groups of class $2$, exponent $p$, and genus $2$ over $K$.
Let $\mathcal{H}_1=\{H_1,\dots, H_s\}$ and $\mathcal{H}_2=\{K_1,\dots, K_t\}$ be fully-refined central decompositions of $G_1$ and $G_2$ respectively.
Then $G_1\cong G_2$ if, and only if, $s=t$ and there exists $M\in\GammaL(2,K)$ and $\sigma$, a permutation of $\{1,\dots, t\}$, such that for all $i$,
\[ \left( \Pf(\circ_{H_i})^M \right) = \left( \Pf(\circ_{K_{i^{\sigma}}})\right). \]
\end{thm}

We want to pause to give two examples that illustrate some of the nuances of Pfaffians.
Our first example shows the need to first determine a fully refined central decomposition of the group, before computing its Pfaffian.
We note that these examples apply more generally for every prime.

\begin{ex}\label{ex:genus}
Let $K$ be a quadratic extension of $\mathbb{F}_5$, and let $G$ be a central product of two copies of $H(K)$ with centers identified.
So $|G|=5^{10}$ and $|G'|=5^2$.
Now let $A=\mathbb{F}_5[x]/\left((x^2+2)^2\right)$ and consider the group $H(A)$.
The ideal $(x^2+2)$ is spanned by $\{x^2+2, x^3+2x\}$ in $A$, and let $S$ be the subspace spanned by $\{ 1, x \}$.
Let
\[ N = \left\{ \begin{bmatrix} 1 & 0 & s \\ & 1 & 0 \\ & & 1 \end{bmatrix} \; \middle| \; s\in S \right\} \triangleleft H(A). \]
Later we prove in Theorem~\ref{thm:genus2-ses} that $H=H(A)/N$ is centrally indecomposable and semi-extraspecial.
Importantly, $|H|=5^{10}$ and $|H'|=5^2$.
Because $G$ is decomposable and $H$ is indecomposable, they are not isomorphic.
Note also that $G$ is genus 1 over $K$ and $H$ is genus 2 over $\mathbb{F}_5$.

We will show they have the same Pfaffians.
Assume $\{\alpha, 1\}$ is a basis for $K$, where $\alpha^2+2=0$.
Set
\[ B = \left[ \begin{array}{c|c}
  \begin{array}{cc} \alpha & 1\\ 3 & \alpha \end{array} & 0 \\ \hline
  0 & \begin{array}{cc} \alpha & 1\\ 3 & \alpha \end{array}
\end{array} \right], \]
so the bimap $\circ_G:\mathbb{F}_5^8\times \mathbb{F}_5^8 \rightarrowtail \mathbb{F}_5^2$ has structure constants $\left[\begin{smallmatrix} 0 & B \\ -B^T & 0 \end{smallmatrix}\right]$.
Replacing the basis of $K$ with indeterminants $X$ and $Y$, the Pfaffian of $G$ is $(X^2+2Y)^2$.
On the other hand, set
\[ C = \begin{bmatrix} 0 & 0 & 2(3x^2+1) & x^3+2x \\
  0 & 2(3x^2+1) & x^3+2x & 3x^2+1 \\
  2(3x^2+1) & x^3+2x & 3x^2+1 & 0 \\
  x^3+2x & 3x^2+1 & 0 & 0 \end{bmatrix}. \]
Assuming the basis for $A$ is $\{1, x, x^2+2, x^3+2x\}$, then the structure constants for the bimap $\circ_H$ is $\left[\begin{smallmatrix} 0 & C \\ -C^T & 0 \end{smallmatrix}\right]$.
Set the basis of the codomain of $\circ_H$ to be $\{ x^3+2x, 3x^2+1 \}$.
Replacing that with indeterminants $X$ and $Y$, the Pfaffian is equal to $(X^2+2Y^2)^2$.
Therefore, the Pfaffians of $\circ_G$ and $\circ_H$ are identical even though $G\not\cong H$.

However, if we instead first determine a fully refined central decomposition for $G$ and $H$ and compute the Pfaffians of the indecomposable factors we get different invariants.
The multiset of Pfaffians for $G$ is $\{ X^2+2Y^2, X^2+2Y^2\}$ (or just $\{X, X\}$ over $K$), and the multiset of Pfaffians for $H$ is $\{ (X^2+2Y^2)^2\}$, which are not equivalent.\qed
\end{ex}

Theorem~\ref{thm:Pfaffian-bijection} greatly depends on the fact that the groups are genus $2$.
When the genus is larger, deciding isomorphism of groups does not reduce to deciding if their Pfaffians are equivalent with respect to Theorem~\ref{thm:Pfaffian-bijection}.
In fact, such groups can have the same Pfaffian and not be isomorphic as the next example illustrates.

\begin{ex}\label{ex:same-Pfaff}
We present two bimaps $V\times V\rightarrowtail W$ with the same Pfaffian that are not pseudo-isometric, so the corresponding groups are not isomorphic.
Suppose $p=3$, then $f(x) = x^3 + 2x + 1$ is irreducible in $\mathbb{F}_3[x]$.
The matrices
\begin{align*}
M_1 &= \begin{bmatrix} 1 & 0 & 0 \\ 0 & 0 & 2 \\ 0 & 2 & 0 \end{bmatrix} &
M_2 &= \begin{bmatrix} 0 & 1 & 0 \\ 1 & 0 & 1 \\ 0 & 1 & 2 \end{bmatrix} &
M_3 &= \begin{bmatrix} 0 & 0 & 1 \\ 0 & 1 & 0 \\ 1 & 0 & 1 \end{bmatrix}
\end{align*}
form the structure constants of the $\mathbb{F}_3$-algebra $\mathbb{F}_3[x]/(f(x))$ with respect to the ordered basis $(1,x,x^2)$.
Let $A$ be the $3\times 3$ matrix with $1$ in the $(1,2)$-entry, $-1$ in the $(2,1)$-entry, and $0$ elsewhere.
Set $V=\mathbb{F}_3^6$ and $W=\mathbb{F}_3^3$.
Define two alternating forms with the following structure constants
\[ \left( \begin{bmatrix} 0 & M_i \\ -M_i & 0 \end{bmatrix} \right) \qquad \text{and} \qquad \left( \begin{bmatrix} 0 & M_i \\ -M_i & A \end{bmatrix} \right). \]

Both bimaps have the same Pfaffians: $x^3+x^2y+xy^2+2xz^2 + 2y^3 + 2y^2z + z^3$.
However, the two bimaps cannot be pseudo-isometric as the former contains at least two 3-dimensional totally isotropic subspaces in $V$, and the latter contains exactly one 3-dimensional totally isotropic subspace.
In the context of $3$-groups, the corresponding groups are not isomorphic as the former contains at least two abelian subgroups of order $3^6$ where the latter contains exactly one such abelian subgroup, see \cite{Lewis:ses-groups} for a proof. \qed
\end{ex}

The multiple threes in Example~\ref{ex:same-Pfaff} are a red herring.
Instead, it illustrates a general phenomenon that happens when the genus is at least 3.
When the genus is 2, $\circ_G$ can be written as a pair of alternating forms.
Kronecker--Dieudonn\'e~\cite{Dieudonne:pairs-mats} and Scharlau \cite{Scharlau:alt-forms} classify indecomposable pairs of alternating forms.
For the (centrally indecomposable) groups coming from Theorem~\ref{thm:indecomp-genus2}(i), there exists bases for $V$ and $W$ such that $\circ_G$ is given by the pair of matrices
\[ \left\{ \begin{bmatrix} 0 & I \\ -I & 0 \end{bmatrix}, \begin{bmatrix} 0 & C \\ -C^T & 0 \end{bmatrix} \right\}, \]
where $C$ is the companion matrix of a power of an irreducible polynomial in $K[x]$.
Together with the fact that the Pfaffian of $\circ_G$ is 0 when $\dim_K(V)$ is odd, this proves the following proposition.

\begin{prop}\label{prop:primary-Pfaff}
Suppose $G$ is exponent $p$ and genus $2$ over $K$.
If $G$ is centrally indecomposable, then the ideal generated by the Pfaffian of $\circ_G:V\times V\rightarrowtail W$ is primary in $K[x,y]$.
\end{prop}

\subsection{$\GammaL(2,K)$ action on the Pfaffians}\label{sec:GL2-action}

The group $\GammaL(2,K)$ acts on polynomials $f(x,y)$ by substitution.
That is, for $M=\left(\left[\begin{smallmatrix} a & b \\ c & d \end{smallmatrix}\right], \sigma\right)\in \GL(2,K)\rtimes\Gal(K)$, set $f^M(x,y) = f^\sigma(ax+by, cx+dy)$, where $\sigma$ acts on the coefficients.
Let $f$ and $g$ be homogeneous polynomials in $x,y$.
Then $f$ is equivalent to $g$ if there exists $k\in K^\times$ and $M\in\GammaL(2,K)$ such that $f^M=k\cdot g$.
In other words, the ideals are equal, $(f^M) = (g)$.
Another perspective is to consider univariate polynomials $f\in K[x]$, with $\deg f=n$, and define
\[ f^M(x)= (cx+d)^n f^\sigma\left(\frac{ax+b}{cx+d}\right). \]
Because the Pfaffians are homogeneous, we consider both of these actions.

The $\GammaL(2,K)$-orbits have been studied by Vaughan-Lee~\cite{Vaughan-Lee:PGL-orbits} and Vishnevetski\u\i~\cite{Vishnevetskii:PGL-orbits} when $K=\mathbb{F}_p$.
We let $N(q,n)$ denote the number of $\GammaL(2, \mathbb{F}_q)$-orbits of monic irreducible polynomials of degree $n$.
One can determine the number of isoclinism classes of the centrally indecomposable groups of genus 2 (over $\mathbb{F}_p$) using \cite{Vaughan-Lee:PGL-orbits} and Theorem~\ref{thm:indecomp}.
Vaughan-Lee's Theorem enumerates the $\GL(2,p)$-orbits of all irreducible monic polynomials in $\mathbb{F}_p[x]$.
Since we do not require the full generality, we record Vaughan-Lee's Theorem for $n\leq 5$.

\begin{thm}[{\cite{Vaughan-Lee:PGL-orbits}}]\label{thm:GL2-orbits}
If $p$ is prime, then
\begin{align*}
  N(p,2) &= 1, & N(p,4) &= \frac{1}{2}\left(p + 2 - \gcd(2, p)\right), \\
  N(p,3) &= 1, & N(p,5) &= \frac{1}{5}\left( p^2 - 2 + \gcd(5,p) + 2\gcd(5, p^2-1)\right) .
\end{align*}
\end{thm}

\section{Semi-extraspecial groups of genus 2}\label{sec:ses-groups-genus2}

We describe the centrally indecomposable, semi-extraspecial groups of genus 2 based on the classification in Theorem~\ref{thm:indecomp-genus2}.
The following lemmas quickly determine which groups of genus 2 are not semi-extraspecial.

\begin{lem}\label{lem:bad-flat}
For all $m\geq 1$, the group $G=H^\flat(K, m)$ is not semi-extraspecial.
\end{lem}

\begin{proof}
As $[G:G']=p^{2m+1}$, the Pfaffian of $\circ_{G}$ is 0, so invoke Theorem~\ref{thm:Pfaffian}.
\end{proof}

In what follows, we characterize the quotients of $H(K[x]/(a(x)^c))$ that are semi-extraspecial groups.
Set $A=K[x]/(a(x)^c)$.
Theorem~\ref{thm:indecomp-genus2}(i) is only concerned with central subgroups $N$ such that $1-N$ is a $K$-subspace of $1-H(A)'$.
Therefore, for a $K$-subspace $S\leq A$, we set
\begin{equation*}
N(S) = \left\{ \begin{bmatrix} 1 & 0 & s \\ & 1 & 0 \\ & & 1 \end{bmatrix} \;\middle|\; s\in S \right\}.
\end{equation*}

\begin{prop}\label{prop:ses-quotients}
Let $a(x) \in K[x]$ be an irreducible polynomial, $c > 1$ be an integer, and $S$ a proper $K$-subspace of $A=K[x]/(a(x)^c)$.
Then $H(A)/N(S)$ is semi-extraspecial if, and only if, $A = S + (a(x)^{c-1})$ and $\deg(a(x))\geq 2$.
\end{prop}

\begin{proof}
Let $I=(a(x)^{c-1})$.
First, we suppose that $A = S + I$.
Since $A$ is a local algebra, the group of units $A^\times$ is $A\setminus (a(x))$.
It follows that every nonzero element of $A$ can be written as $ua^e$, for $u\in A^\times$ and $0\leq e\leq c-1$.
Since $S$ is a proper subspace of $A$, it follows that $\emptyset \ne A\setminus S\subseteq I$.

To show that $H(A)/N(S)$ is semi-extraspecial, it is enough to show that for every $0\ne f\in A$, and $h+S \in A/S$, there exists $g\in A$ such that $fg=h$.
Let $f=ua^e$, for some $u\in A^\times$ and $0\leq e\leq c-1$.
Without loss of generality, $h+S\in A/S$ can be written as $va^{c-1} + S$, for some $v\in A^\times$.
With $g=u^{-1}va^{c-e-1}$, we have that $fg=h$, and so $H(A)/N(S)$ is semi-extraspecial.

On the other hand, suppose $S+I \ne A$.
Let $h(x)\in A\setminus (S+I)$.
Therefore, for all $f(x)\in I$ and for all $g(x)\in A$, $fg\in I$.
In particular, $fg+S\ne h+S$, so $H(A)/N(S)$ cannot be semi-extraspecial.
\end{proof}

We summarize the above results in the following theorem to classify all indecomposable, semi-extraspecial groups of genus 2 with exponent $p$.

% Fixed to accurately state when G is indecomposable.

\begin{thm}\label{thm:genus2-ses}
Suppose $G$ is class $2$, exponent $p$, and genus $2$ over $K$.
Then $G$ is centrally indecomposable and semi-extraspecial if, and only if, there exists an irreducible polynomial $a(x)\in K[x]$ of degree $\geq 2$, a positive integer $c$, and a subspace $S\leq A := K[x]/(a(x)^c)$ with $2$-dimensional complement $\langle f,g\rangle$, such that 
\begin{enumerate}
\item[$(i)$] $A = S + (a(x)^{c-1})$,
\item[$(ii)$] $A$ is an irreducible $K[f^{-1}g]$-module, and
\item[$(iii)$] $G\cong H(A) / N(S)$.
\end{enumerate}
\end{thm}

\begin{proof}
If $G$ is indecomposable and semi-extraspecial, then by Theorem~\ref{thm:indecomp-genus2} $G$ is contained in one of two (disjoint) families of genus 2 groups.
By Lemma~\ref{lem:bad-flat} property (iii) holds, and by Proposition~\ref{prop:ses-quotients}, property (i) holds. 
Therefore, there exists $f,g\in A^\times$ that span a complement to $S$ in $A$. 
As $G$ is indecomposable, by the proof of \cite{BMW}*{Theorem~3.18}, $A$ is an irreducible $K[f^{-1}g]$-module. 

If, on the other hand, $G$ satisfies properties (i)--(iii) above, then by Proposition~\ref{prop:ses-quotients}, $G$ is semi-extraspecial.
As $\langle f, g\rangle \oplus S = A$, $\circ_G$ can be written as a pair of alternating matrices with respect to $f$ and $g$, as in Section~\ref{sec:Pfaffians}, of the form
\[ \left\{ \begin{bmatrix} 0 & M_f \\ -M_f^T & 0 \end{bmatrix}, \begin{bmatrix} 0 & M_g \\ -M_g^T & 0 \end{bmatrix} \right\}. \]
Since $G$ is semi-extraspecial, $M_f$ is invertible, so the block-diagonal matrix $M_f^{-1}\oplus I_{cd}$ induces an isometry of $\circ_G$ to an alternating bimap with $I_{cd}$ and $M_f^{-1}M_g$ in the top-right corners. 
Because $A$ is an irreducible $K[f^{-1}g]$-module, it follows that the rational canonical form of $M_f^{-1}M_g$ is the companion matrix of a primary polynomial.
Thus, $G$ is indecomposable.
\end{proof}

\section{Enumerating isoclinism classes}

We use Theorems~\ref{thm:Pfaffian-bijection} and~\ref{thm:genus2-ses} to count the isoclinism classes of semi-extraspecial groups whose derived subgroup has order $p^2$.
Therefore, these groups can have genus $1$ or $2$ over $\mathbb{F}_p$.
The genus 1 case essentially follows from the classification of generalized Heisenberg groups in \cite{LW:Heisenberg}*{Theorem~3.1}.

In the next lemma, we use a consequence of Witt's Theorem \cite{Artin:Geometric-algebra}*{Chapter~3}: two alternating bilinear $K$-forms are isometric if, and only if, their radicals have the same dimension.

\begin{lem}\label{lem:genus-1}
Assume $n\geq 2$ and $q$ is a $p$-power. There are
\[ H(q,n+1) = \left\{ \begin{array}{ll} 1 & n \text{ even},\\ 0 & \text{otherwise}, \end{array}\right. \]
semi-extraspecial groups of order $q^{n+1}$ with exponent $p$ and derived subgroup of order $q$, where $\mathcal{C}(\circ_G)\cong\mathbb{F}_{q}$.
\end{lem}

\begin{proof}
Let $G$ be a group satisfying the assumptions and $K=\mathcal{C}(\circ_G)$.
As $\circ_G:V\times V\rightarrowtail W$ is $K$-bilinear and $\dim_{K}(W)=1$, it follows that $\circ_G$ is an alternating $K$-form.
By Lemma~\ref{lem:direct-decomp}, $G$ is directly indecomposable, so $\circ_G$ has trivial radical.
This leads to a contradiction if $\dim_K(V)$ is odd, so no such groups exist.
On the other hand, if $\dim_K(V)=2n$, then
\[ G \cong \left\{ \begin{bmatrix} 1 & u & w \\ & I_n & v^T \\ & & 1 \end{bmatrix} \;\middle|\; \begin{array}{c} u,v\in K^n, \\ w \in K \end{array} \right\} \leq \GL(n+2, K). \qedhere \]
\end{proof}

To prove the next theorem, we require a way to build nilpotent groups from bimaps.
Suppose $\circ : U \times V \rightarrowtail W$ is a $K$-bimap.
Then the Heisenberg group with respect to $\circ$ is the group
\[ H(\circ) = \left\{ \begin{bmatrix} 1 & u & w \\ & 1 & v \\ & & 1 \end{bmatrix} \;\middle|\; u\in U, v\in V, w\in W \right\}, \]
where the product is given by matrix multiplication using $\circ$.
Assuming $\circ$ is fully nondegenerate, the center of $H(\circ)$ equals the derived subgroup and $Z(H(\circ))\cong W$.
Moreover, the commutator bimap of $H(\circ)$ has ``doubled'' $[,]:(U\oplus V)\times (U\oplus V)\rightarrowtail W$ where
\[ \big( (u_1, v_1), (u_2, v_2) \big) \mapsto u_1\circ v_2 - u_2\circ v_1. \]
As matrices, if $\circ$ has structure constants given by the sequence $(M_i)_{i=1}^d$, then $[,]$ has structure constants of the form $\left[ \begin{smallmatrix} 0 & M_i \\ -M_i^T & 0 \end{smallmatrix} \right]$.

Recall from Section~\ref{sec:GL2-action} that $N(q,n)$ is the number of $\GammaL(2, \mathbb{F}_q)$-orbits of monic irreducible polynomials of degree $n$.
As $\GL(2, \mathbb{F}_q)$ is transitive on the nonzero vectors in $\mathbb{F}_q^2$, it follows that $N(q, 1)=1$.

\begin{thm}\label{thm:indecomp}
Suppose $n\geq 2$ and $q$ is a $p$-power.
There are $I(q, 2n+2) = -1 + \sum_{d\mid n} N(q,d)$ pairwise non-isomorphic, centrally indecomposable, semi-extraspecial groups of order $q^{2n+2}$ with exponent $p$ and genus $2$ over $\mathbb{F}_q$.
\end{thm}

\begin{proof}
First apply Theorem~\ref{thm:Pfaffian-bijection} to establish that two indecomposable genus 2 groups with exponent $p$ are isomorphic if, and only if, their Pfaffians are in the same $\GammaL(2,q)$-orbit.
By Proposition~\ref{prop:primary-Pfaff}, their Pfaffians are primary.
Let $a(x)^c=x^{cd} + a_{cd-1}x^{cd-1} + \cdots + a_0$, and assume $a(x)$ is an irreducible polynomial in $\mathbb{F}_q[x]$ and $c\geq 1$.
Let $C$ be the companion matrix of $a(x)^c$:
\[ C = \left[\begin{array}{c|c}
  \begin{array}{c} 0 \\ \vdots \\ 0 \end{array} & I_{cd-1} \\ \hline
  -a_0 & \begin{array}{ccc} -a_1 & \cdots & -a_{cd-1} \end{array} \end{array}\right]. \]
The group with commutator bimap whose structure constants are
\[ \left( \begin{bmatrix} 0 & I \\ -I & 0 \end{bmatrix}, \; \begin{bmatrix} 0 & -C \\ C^T & 0 \end{bmatrix} \right) \]
has a Pfaffian equal to $X^{cd} + a_{cd-1}X^{cd-1}Y + \cdots + a_0Y^{cd}$.
Note that this is the determinant of $IX - CY$.
Therefore, for every irreducible polynomial $a(x)$ of degree $d$ and $c\geq 1$, there exists an indecomposable $p$-group of order $q^{2cd+2}$ of genus 2 and exponent $p$ where $\Pf(\circ_G)=a(x)^c$.
By Theorem~\ref{thm:Pfaffian-bijection}, there are $N(q,d)$ isoclinism classes of this type.
By Theorem~\ref{thm:Pfaffian}, semi-extraspecial groups have no linear factors, so we require $d\geq 2$.
Thus, to get the number of isoclinism classes of indecomposable semi-extraspecial groups of order $q^{2n+2}$ that are genus 2 over $\mathbb{F}_q$, we sum over all the divisors $d$ of $n$, excluding $d=1$.
\end{proof}

\begin{thm}\label{thm:p^8}
For every prime $p$, there is exactly one isoclinism class of semi-extraspecial groups $G$ with the property that $|G|=p^8$ and $|G'|=p^2$.
\end{thm}

\begin{proof}
First, we show that $G$ must be indecomposable.
If $\{H_1,H_2\}$ is a central decomposition of $G$, then $|H_i|\in\{p^4, p^5, p^6\}$ because $[G:G']=p^6$.
From Proposition~\ref{prop:central-decomp}, $Z(H_i)=Z(G)\cong \mathbb{Z}_p^2$.
There are no genus 2 groups of order $p^4$ and by Lemma~\ref{lem:bad-flat}, the genus 2 group of order $p^5$ is not semi-extraspecial.
Therefore, $G$ is indecomposable.
By Theorems~\ref{thm:GL2-orbits} and~\ref{thm:indecomp}, there is $I(p, 8)=N(p,3)=1$ isoclinism class.
\end{proof}

\begin{remark}\label{remark:genus3}
An interesting consequence of Theorem~\ref{thm:p^8} is that every semi-extraspecial group with exponent $p$ and of order $p^9$ with $|G'|=p^3$, has exactly two non-trivial central nonabelian quotients, determined solely by the order of the factor group.
In particular, heuristics to construct characteristic subgroups via graph colorings, like in~\cite{BOW:graded-algebras}, or from algebras associated to $\circ_G$, like in \cites{M:efficient-filters,Wilson:filters}, produce no new information.
These groups may pose a formidable challenge to isomorphism testing.
\end{remark}

Are there other semi-extraspecial groups, with derived subgroup of order $p^3$, that have the property in Remark~\ref{remark:genus3}?
That is, if $N, M \leq Z(G)$ and $|N|=|M|$, then $G/N$ and $G/M$ are isoclinic.
If there are no other such groups, then what is the lower bound on the number of possible isoclinism classes of the form $G/N$, for a fixed order $|N|$?
On the other end of the spectrum, what is the upper bound?

\subsection{Proof of Theorem~\ref{main2}}

Enumerating semi-extraspecial groups of larger order is not as simple as in Theorem~\ref{thm:p^8} and will require understanding how $\GL(2,p)$ acts on (principal) ideals of \emph{reducible} polynomials.
We now focus on univariate polynomials in $\mathbb{F}_p[x]$, so we dehomogenize $\Pf(\circ_G)$ by setting $y=1$.
If $f\in \mathbb{F}_p[x]$, we let $S_f$ denote the subgroup, $\Stab_{\GL(2,p)}((f))$, stabilizing the ideal $(f)$.

The next lemma is a well-known statement coming from the (maximal) subgroup structure of $\PSL(2,p)$.
We refer the reader to Dickson \cite{Dickson}, see also Huppert \cite{Huppert} and, for historical remarks, King \cite{King}*{Theorem~2.1}.

\begin{lem}\label{lem:maximal-stab}
If $f$ is an irreducible quadratic in $\mathbb{F}_p[x]$, then $S_f\cong\GammaL(1,p^2)$.
The subgroups isomorphic to $\GammaL(1,p^2)$ in $\GL(2,p)$ are in bijection to monic irreducible quadratics in $\mathbb{F}_p[x]$.
\end{lem}

\begin{proof}
By orbit counting, $|S_f|=2(p^2-1)$.
For $\alpha\in \mathbb{F}_{p^2}$ and $M=\left[\begin{smallmatrix} a & b \\ c & d \end{smallmatrix}\right]$, if $f(x)=(x-\alpha )(x-\alpha^p)$, then
\[ f^M(x) = (a-\alpha c)(a-\alpha^pc)\left( x - \dfrac{\alpha d-b}{a-\alpha c} \right)\left( x - \dfrac{\alpha^p d-b}{a-\alpha^p c} \right) .\]
This is the inverse action on the roots of $f$, so the lemma follows. %\cite{King}.
\end{proof}

Note that the kernel of the action of $\GL(2,p)$ on polynomials is its center, so we work with $\PGL(2,p)$ instead.
In order to enumerate the isoclinism classes of semi-extraspecial groups order $p^{10}$ with derived subgroup of order $p^2$, we need to understand an action of dihedral groups in $\PGL(2,p)$.
Assuming $p$ is an odd prime, let $\Delta$ be the set of maximal dihedral subgroups of $\PGL(2,p)$ of order $2(p+1)$.
That is, $\Delta$ contains all the images of the subgroups isomorphic to $\GammaL(1,p^2)$ in $\GL(2,p)$.
Using the bijection given by Lemma~\ref{lem:maximal-stab}, a dihedral group acts on $\Delta$ in the following way.
If $D\in\Delta$ and $f$ is an irreducible quadratic associated to $D_f\in \Delta$, then for $\delta\in D$, $D_f^\delta=D_{f^\delta}$.
Therefore, for $D\in\Delta$, the $D$-orbit of $D$ is trivial.

\begin{lem}\label{lem:dihedral-orbits}
Let $p$ be an odd prime.
For a fixed $D\in\Delta$, the number of $D$-orbits on $\Delta-D$ is $(p-1)/2$.
\end{lem}

\begin{proof}
We count these orbits by counting the fixed points for every $\delta\in D$.
First note that if $1\ne\delta\in D$ is not an involution, then $N_{\PGL(2,p)}(\langle\delta\rangle)=D$.
Thus, such a $\delta$ cannot fix another dihedral group in $\Delta$, so we focus only on the involutions of $D$.

Each dihedral group in $\Delta$ intersects $\PSL(2,p)$ uniquely in a dihedral group of order $p+1$.
Let $\overline{D}=D\cap\PSL(2,p)$.
Up to conjugacy, there are two types of involutions in $D$: those contained in $\PSL(2,p)$ and those outside \cite{Steinberg:rep-PGL}*{pp.~226--227}.
Suppose $p\equiv 1$ mod $4$, so $\overline{D}$ has exactly $(p+1)/2$ involutions.
Since $|\Delta|=(p^2-p)/2$, there are $(p^2-p)(p+1)/4$ involutions coming up this way.
However, $\PSL(2,p)$ has exactly $p(p+1)/2$ involutions \cite{Dornhoff:rep-theory}*{Chapter~38}, so each involution in $\PSL(2,p)$ lies in $(p-1)/2$ dihedral subgroups of order $p+1$.
Therefore, the involutions of $\overline{D}$ stabilize $(p-3)/2$ dihedral groups in $\Delta-D$.

There are $(p+3)/2$ involutions in $D$ outside of $\PSL(2,p)$, and as many as $(p^2-p)(p+3)/4$ involutions outside of $\PSL(2,p)$.
One of these involutions, $\delta$,  lies in $Z(D)$, so its centralizer is $D$.
As there are $(p^2-p)/2$ involutions of $\PGL(2,p)$ outside of $\PSL(2,p)$ \cite{Dornhoff:rep-theory}, it follows that $\delta$ fixes $(p+3)/2-1$ other dihedral groups in $\Delta-D$.
Therefore, the number of orbits in this case is equal to
\[ \dfrac{\left(\frac{1}{2}(p^2-p)-1\right) + \left(\frac{1}{2}(p+1)\frac{1}{2}(p-3)\right) + \left(\frac{1}{2}(p+3) \frac{1}{2}(p+1)\right)}{2(p+1)} = \dfrac{p-1}{2}. \]

The case when $p\equiv 3$ mod $4$ is similar to the above case.
There is one more involution in $\PSL(2,p)$ than above, namely $\overline{D}$ contains a central involution.
The number of fixed points for these involutions follows a similar argument as above.
The argument for the involutions outside of $\PSL(2,p)$ is slightly different.
Of the $(p+1)/2$ involutions of $D$ outside of $\PSL(2,p)$, each one is contained in the center of some dihedral group of order $2(p-1)$, so its centralizer is that dihedral group.
Hence, there are $(p^2-p)/2$ such involutions, which are all of the involutions of $\PGL(2,p)$ outside $\PSL(2,p)$.
Therefore, each involution of this type is contained in $(p-1)/2$ groups in $\Delta$, and thus, fixes $(p-3)/2$ groups in $\Delta-D$.
\end{proof}

Now we are in a position to prove Theorem~\ref{main2}.

\begin{proof}[Proof of Theorem~\ref{main2}]
If $G$ is indecomposable, then apply Theorems~\ref{thm:GL2-orbits} and~\ref{thm:indecomp}, so there are
\[ I(p, 2\cdot 4+2) = N(p,4)+N(p,2)= \frac{1}{2}\left( p + 4 - \gcd(2, p)\right) \]
such isoclinism classes.
Now suppose $G$ is centrally decomposable with fully refined central decomposition $\mathcal{H}$.
From the proof of Theorem~\ref{thm:p^8}, for each $H\in\mathcal{H}$, $|H|\geq p^6$.
Because $|G|=p^{10}$, it follows that $\mathcal{H}=\{H_1,H_2\}$, where $|H_i|=p^6$.
We have two options for the centroid of $\circ_G$ from Proposition~\ref{prop:centroid-field}: $\mathcal{C}(\circ_G)$ is isomorphic to either $\mathbb{F}_p$ or $\mathbb{F}_{p^2}$.
If $\mathcal{C}(\circ_G)$ is the quadratic extension, then by Lemma~\ref{lem:genus-1}, there is exactly $H(p^2, 5)=1$ group with this property.

We suppose the centroid is the prime field.
In this case, the Pfaffian of $\circ_G$ in $\mathbb{F}_p[x,y]$ is the product of two homogeneous quadratics $f$ and $g$.
Without loss of generality, we assume that both $f$ and $g$ are monic.
If $f=g$, then $\circ_G$ is $\mathbb{F}_p[x]/(f)$-bilinear, so the centroid would be a (non-trivial) field extension, which has already been accounted for.
Therefore, $f\ne g$.
This means that, for $p=2$, there are no groups of this type because there is a unique irreducible monic quadratic.
Thus, we now assume $p>2$.
By Theorem~\ref{thm:Pfaffian-bijection}, the number of isoclinism classes of these groups is the number of $\PGL(2,p)$-orbits on the ideals generated by a product of monic irreducible quadratics.
As $\PGL(2,p)$ is transitive on monic irreducible quadratics, the number of orbits is equal to the number of orbits of one of the stabilizers, say, $\Stab_{\PGL(2,p)}((f))$.
By Lemmas~\ref{lem:maximal-stab} and~\ref{lem:dihedral-orbits}, there are exactly $(p-1)/2$ orbits.
Thus, the statement of the theorem holds.
\end{proof}

\subsection{Proof of Theorem~\ref{main3}}

Because all the monic irreducible cubics lie in one $\PGL(2,p)$-orbit, c.f.\ Theorem~\ref{thm:GL2-orbits}, it follows, again by orbit counting, that their stabilizers have order $3$ in $\PGL(2,p)$.

\begin{lem}\label{lem:Syl-3}
In $\PGL(2,p)$, if $p\equiv 2\mod 3$, then there is a bijection between the Sylow $3$-subgroups and the dihedral subgroups of order $2(p+1)$ given by inclusion.
\end{lem}

\begin{proof}
Let $D$ denote a dihedral subgroup of $G=\PGL(2,p)$ of order $2(p+1)$.
As $p\equiv 2\mod 3$, $D$ contains a Sylow $3$-subgroup, $P$.
If $x\in D$ has order $p+1$, then $P\leq \langle x\rangle$.
Because $N_{G}(\langle x\rangle) = D$, it follows that $N_{G}(P)=D$.
\end{proof}

\begin{proof}[Proof of Theorem~\ref{main3}]
In the indecomposable case, we apply Theorem~\ref{thm:indecomp}, and in the case where $G$ is genus 2 over the quadratic extension, we apply Lemma~\ref{lem:genus-1}.
By Theorem~\ref{thm:GL2-orbits}, the number of isoclinism classes is
\[ I(p, 2\cdot 5 + 2) + H(p^2, 6) = N(p,5) + 0 = \frac{1}{5} \left( p^2-2+\gcd(5,p)+2\gcd(5, p^2-1)\right). \]

Now, all we have left to do is to count the number isoclinism classes in the decomposable case where the centroid is the prime field.
By Proposition~\ref{prop:central-decomp} and Theorem~\ref{thm:genus2-ses}, a fully refined central decomposition of $G$ is $\{H_1,H_2\}$, where $|H_1|=p^6$ and $|H_2|=p^8$.
Thus, the multiset of Pfaffians contains exactly one quadratic and cubic.
Hence, the stabilizer, in $\GL(2,p)$, of the multiset is the intersection of stabilizers of a quadratic and a cubic.

By Lemma~\ref{lem:maximal-stab} the quadratic stabilizer is isomorphic to $\GammaL(1,p^2)$.
Since $N(p,3)=1$, the stabilizer of a cubic has order $3(p-1)$ as there are $(p^3-p)/3$ monic irreducible cubics in $\mathbb{F}_p[x]$.
In $\PGL(2,p)$, this is the intersection of the dihedral group of order $2p+2$ with a cyclic group of order 3.
Therefore, when $p\equiv 0,1\mod 3$, these intersections are always trivial in $\PGL(2,p)$, and the number of orbits in this case is
\[ \dfrac{\frac{1}{2}(p^2-p)\frac{1}{3}(p^3-p)}{(p^3-p)} = \dfrac{1}{6} (p^2-p). \]

When $p\equiv 2 \mod 3$, Lemma~\ref{lem:Syl-3} implies that the Sylow 3-subgroups are cyclic and so contain a unique subgroup of order 3.
As there are $(p^3-p)/3$ irreducible cubics and $(p^2-p)/2$ subgroups of order 3 in $\PGL(2,p)$, it follows that $2(p+1)/3$ cubics have identical stabilizers in $\PGL(2,p)$.
Hence, for a fixed irreducible quadratic there are $2(p+1)/3$ irreducible cubics such that their stabilizer has order 3;
for the remaining $(p^3-3p-2)/3$ cubics, their stabilizer is trivial.
Therefore, the number of orbits is given by
\[ \dfrac{\frac{1}{2}(p^2-p)\left(2(p+1) + \frac{1}{3}(p^3-3p-2)\right)}{(p^3-p)} = \dfrac{1}{6}(p^2-p+4). \qedhere\]
\end{proof}

\section*{Acknowledgements}

We thank Mikhail A.\ Chebotar and Jenya Soprunova for translating parts of \cite{Vishnevetskii:PGL-orbits}.


\begin{bibdiv}
\begin{biblist}

\bib{Artin:Geometric-algebra}{book}{
   author={Artin, E.},
   title={Geometric algebra},
   publisher={Interscience Publishers, Inc., New York-London},
   date={1957},
   pages={x+214},
   review={\MR{0082463}},
}

\bib{Beisiegel:ses-groups}{article}{
   author={Beisiegel, Bert},
   title={Semi-extraspezielle $p$-Gruppen},
   language={German},
   journal={Math. Z.},
   volume={156},
   date={1977},
   number={3},
   pages={247--254},
   issn={0025-5874},
   review={\MR{0473004}},
}

\bib{BMW}{article}{
   author={Brooksbank, Peter A.},
   author={Maglione, Joshua},
   author={Wilson, James B.},
   title={A fast isomorphism test for groups whose Lie algebra has genus 2},
   journal={J. Algebra},
   volume={473},
   date={2017},
   pages={545--590},
   issn={0021-8693},
   review={\MR{3591162}},
}
%
%\bib{BMW:Heisenberg-groups}{unpublished}{
%   author={Brooksbank, Peter A.},
%   author={Maglione, Joshua},
%   author={Wilson, James B.},
%   title={Central quotients of Heisenberg groups over local algebras},
%   status={in preparation},
%}

\bib{BOW:graded-algebras}{unpublished}{
   author={Brooksbank, Peter A.},
   author={O'Brien, E.A.},
   author={Wilson, James B.},
   title={A fast isomorphism test for groups whose Lie algebra has genus 2},
   status={preprint},
   note={\url{arXiv:1708.08873}},
}

\bib{Camina}{article}{
   author={Camina, A. R.},
   title={Some conditions which almost characterize Frobenius groups},
   journal={Israel J. Math.},
   volume={31},
   date={1978},
   number={2},
   pages={153--160},
   issn={0021-2172},
   review={\MR{516251}},
}

\bib{CMS:Frobenius-groups}{article}{
   author={Chillag, David},
   author={Mann, Avinoam},
   author={Scoppola, Carlo M.},
   title={Generalized Frobenius groups. II},
   journal={Israel J. Math.},
   volume={62},
   date={1988},
   number={3},
   pages={269--282},
   issn={0021-2172},
   review={\MR{955132}},
}

\bib{Dempwolff:81}{article}{
   author={Dempwolff, Ulrich},
   title={Semifield planes of order 81},
   journal={J. Geom.},
   volume={89},
   date={2008},
   number={1-2},
   pages={1--16},
   issn={0047-2468},
   review={\MR{2457018}},
}

\bib{Dickson}{book}{
   author={Dickson, Leonard Eugene},
   title={Linear groups: With an exposition of the Galois field theory},
   series={with an introduction by W. Magnus},
   publisher={Dover Publications, Inc., New York},
   date={1958},
   pages={xvi+312},
   review={\MR{0104735}},
}

\bib{Dieudonne:pairs-mats}{article}{
   author={Dieudonn\'e, Jean},
   title={Sur la r\'eduction canonique des couples de matrices},
   language={French},
   journal={Bull. Soc. Math. France},
   volume={74},
   date={1946},
   pages={130--146},
   issn={0037-9484},
   review={\MR{0022826}},
}

\bib{Dornhoff:rep-theory}{book}{
   author={Dornhoff, Larry},
   title={Group representation theory. Part A: Ordinary representation
   theory},
   note={Pure and Applied Mathematics, 7},
   publisher={Marcel Dekker, Inc., New York},
   date={1971},
   pages={vii+pp. 1--254},
   review={\MR{0347959}},
}

\bib{Hall:isoclinism}{article}{
   author={Hall, P.},
   title={The classification of prime-power groups},
   journal={J. Reine Angew. Math.},
   volume={182},
   date={1940},
   pages={130--141},
   issn={0075-4102},
   review={\MR{0003389}},
}

\bib{Huppert}{book}{
   author={Huppert, B.},
   title={Endliche Gruppen. I},
   language={German},
   series={Die Grundlehren der Mathematischen Wissenschaften, Band 134},
   publisher={Springer-Verlag, Berlin-New York},
   date={1967},
   pages={xii+793},
   review={\MR{0224703}},
}

\bib{JJB:Handbook}{book}{
   author={Johnson, Norman L.},
   author={Jha, Vikram},
   author={Biliotti, Mauro},
   title={Handbook of finite translation planes},
   series={Pure and Applied Mathematics (Boca Raton)},
   volume={289},
   publisher={Chapman \& Hall/CRC, Boca Raton, FL},
   date={2007},
   pages={xxii+861},
   isbn={978-1-58488-605-1},
   isbn={1-58488-605-6},
   review={\MR{2290291}},
}

\bib{Kantor:finite-semifields}{article}{
   author={Kantor, William M.},
   title={Finite semifields},
   conference={
      title={Finite geometries, groups, and computation},
   },
   book={
      publisher={Walter de Gruyter, Berlin},
   },
   date={2006},
   pages={103--114},
   review={\MR{2258004}},
}

\bib{King}{article}{
   author={King, Oliver H.},
   title={The subgroup structure of finite classical groups in terms of
   geometric configurations},
   conference={
      title={Surveys in combinatorics 2005},
   },
   book={
      series={London Math. Soc. Lecture Note Ser.},
      volume={327},
      publisher={Cambridge Univ. Press, Cambridge},
   },
   date={2005},
   pages={29--56},
   review={\MR{2187733}},
}

\bib{Lewis:expository}{unpublished}{
   author={Lewis, Mark L.},
   title={Semi-extraspecial groups},
   status={submitted},
   note={\url{arXiv:1709.03857}},
}

\bib{Lewis:ses-groups}{unpublished}{
   author={Lewis, Mark L.},
   title={Semi-extraspecial groups with an abelian subgroup of maximal possible order},
   status={submitted},
   note={\url{arXiv:1710.10299}},
}

\bib{LW:Heisenberg}{article}{
   author={Lewis, Mark L.},
   author={Wilson, James B.},
   title={Isomorphism in expanding families of indistinguishable groups},
   journal={Groups Complex. Cryptol.},
   volume={4},
   date={2012},
   number={1},
   pages={73--110},
   issn={1867-1144},
   review={\MR{2921156}},
}

\bib{Macdonald:semi-extraspecial}{article}{
   author={Macdonald, I. D.},
   title={Some $p$-groups of Frobenius and extra-special type},
   journal={Israel J. Math.},
   volume={40},
   date={1981},
   number={3-4},
   pages={350--364 (1982)},
   issn={0021-2172},
   review={\MR{654591}},
}

\bib{M:efficient-filters}{article}{
   author={Maglione, Joshua},
   title={Efficient characteristic refinements for finite groups},
   journal={J. Symbolic Comput.},
   volume={80},
   date={2017},
   number={part 2},
   part={part 2},
   pages={511--520},
   issn={0747-7171},
   review={\MR{3574524}},
}

\bib{Scharlau:alt-forms}{article}{
   author={Scharlau, Rudolf},
   title={Paare alternierender Formen},
   journal={Math. Z.},
   volume={147},
   date={1976},
   number={1},
   pages={13--19},
   issn={0025-5874},
   review={\MR{0419484}},
}

\bib{Steinberg:rep-PGL}{article}{
   author={Steinberg, Robert},
   title={The representations of ${\rm GL}(3,q), {\rm GL} (4,q), {\rm PGL}
   (3,q)$, and ${\rm PGL} (4,q)$},
   journal={Canadian J. Math.},
   volume={3},
   date={1951},
   pages={225--235},
   issn={0008-414X},
   review={\MR{0041851}},
}

\bib{Vaughan-Lee:PGL-orbits}{unpublished}{
   author={Vaughan-Lee, Michael},
   title={Orbits of irreducible binary forms over GF($p$)},
   status={preprint},
   note={\url{arXiv:1705.07418}},
}

\bib{Vishnevetskii:PGL-orbits}{article}{
   author={Vishnevetski\u\i , A. L.},
   title={The number of classes of projectively equivalent binary forms over
   a finite field},
   language={Russian, with English summary},
   journal={Dokl. Akad. Nauk Ukrain. SSR Ser. A},
   date={1982},
   number={4},
   pages={9--12, 86},
   issn={0201-8446},
   review={\MR{659930}},
}

\bib{Verardi}{article}{
   author={Verardi, Libero},
   title={Semi-extraspecial groups of exponent $p$},
   language={Italian, with English summary},
   journal={Ann. Mat. Pura Appl. (4)},
   volume={148},
   date={1987},
   pages={131--171},
   issn={0003-4622},
   review={\MR{932762}},
}

\bib{Wilson:direct-decomposition}{article}{
   author={Wilson, James B.},
   title={Existence, algorithms, and asymptotics of direct product
   decompositions, I},
   journal={Groups Complex. Cryptol.},
   volume={4},
   date={2012},
   number={1},
   pages={33--72},
   issn={1867-1144},
   review={\MR{2921155}},
}

\bib{Wilson:filters}{article}{
   author={Wilson, James B.},
   title={More characteristic subgroups, Lie rings, and isomorphism tests
   for $p$-groups},
   journal={J. Group Theory},
   volume={16},
   date={2013},
   number={6},
   pages={875--897},
   issn={1433-5883},
   review={\MR{3198722}},
}

\end{biblist}
\end{bibdiv}
\end{document}